\documentclass[letterpaper,11pt]{amsart}
\usepackage[margin=2.4cm]{geometry}

\usepackage[utf8]{inputenc}
\usepackage[english]{babel}
\usepackage{amsmath,amssymb,amsfonts,amsthm}
\usepackage{graphicx}
\usepackage{hyperref}
\usepackage[draft]{fixme}
\newcounter{thm}

\newtheorem{theorem}[thm]{Theorem}
\newtheorem*{theorem*}{Theorem}
\newtheorem{lemma}[thm]{Lemma}

\newtheorem*{conjecture*}{Conjecture}
\newtheorem{corollary}[thm]{Corollary}

\newtheorem{remark}[thm]{Remark}
\theoremstyle{definition}
\newtheorem{definition}[thm]{Definition}

\newcommand{\distkr}{\operatorname{dist}_\text{\rm KR}}

\newcommand{\eps}{\varepsilon}
\newcommand{\id}{\operatorname{id}}

\newcommand{\bbC}{\mathbb C}

\newcommand{\bbR}{\mathbb R}

\newcommand{\bbQ}{\mathbb Q}

\newcommand{\bbN}{\mathbb N}

\newcommand{\bbZ}{\mathbb Z}

\title{Authomorphic measures with negative exponents for multicritical circle maps}
\author{Nataliya Goncharuk}

\address{Nataliya Goncharuk (corresponding author), Texas A\&M University, College Station, TX, USA}
\email{\url{natasha_goncharuk@tamu.edu}}
\author {Michael Yampolsky}
\address{Michael Yampolsky, University of Toronto, Toronto, Canada}
\email{\url{yampol@math.toronto.edu}}

\begin{document}
\begin{abstract}
Authomorphic or $s$-measures for circle diffeomorphisms were introduced by R.~Douady and J.-C.~Yoccoz \cite{DY}. They 
have multiple applications in circle dynamics, with the case $s=-1$ being particularly important for describing conjugacy classes.
In \cite{FGN}, the authors proved existence and uniqueness of automorphic $s$-measures  for multicritical circle maps for all $s>0$. 
The purpose of this paper is to extend these results to $s$-measures with negative values of $s$. As an application, we prove a smoothness result for irrational Arnold tongues in families of multicritical maps.
\end{abstract}

\maketitle

\section{Introduction}
Recall that the rotation number of an orientation-preserving circle homeomorphism $f$ is given by 
$$
\rho(f) = \lim_{n\to \infty}\frac{F^n(x)}{n}, \quad \rho(f) \in \bbR/\bbZ,
$$
where  $x\in \bbR/\bbZ$ and $F\colon \bbR \to \bbR$ is a lift of $f\colon \bbR/\bbZ\to \bbR/\bbZ$ to the real line. 

For a circle diffeomorphism $f:\bbR/\bbZ\to\bbR/\bbZ$ a Borel probability measure $\mu$ on the circle is called automorphic with an exponent $s$, or simply an $s$-measure, if, for any continuous test function $\xi$, one has
\begin{equation}
\label{eq-def0}
\int_{\bbR/\bbZ}\xi d\mu = \int_{\bbR/\bbZ} \xi(f(x))(f'(x))^s d\mu.
\end{equation}
In \cite{DY}, R.~Doaudy and J.-C.~Yoccoz showed that if the rotation number $\rho(f)$ is irrational, then for each $s$, such a measure exists and is unique.

If a circle homeomorphism $f$ is smooth except for finitely many critical points, the equation (\ref{eq-def0}) makes sense for $s\geq 0$. In the case when all critical points of $f$ are non-flat (such a map is known as a multicritical circle map), the existence and uniqueness of an $s$-measure was recently established by E.~de~Faria, P.~Guarino and B.~Nussenzveig \cite{FGN}.

For negative values of $s$, the invariance equation (\ref{eq-def0}) needs to be rewritten:

\begin{definition}
For any circle homeomorphism $f\colon \bbR/\bbZ\to \bbR/\bbZ$ which is smooth except for finitely many points and $s<0$, a Borel probability measure $\mu$ on a circle is called an $s$-measure if, for any continuous test function $\phi$, we have 
\begin{equation}
\label{eq-def}
\int_{\bbR/\bbZ}\phi d\mu = \int_{\bbR/\bbZ} (f'|_{f^{-1}(x)})^{-s} \phi|_{f^{-1}(x)} d\mu.
\end{equation}
\end{definition}
This is equivalent to (\ref{eq-def0}) in the above described cases, if one uses test functions of the form $$\xi(x)=\phi(f^{-1}(x))(f'(f^{-1}(x)))^{-s}.$$
By iterating the relation (\ref{eq-def}), we get, for all $k\in\bbN$:
\begin{equation}
\label{eq-minus1}
\int_{\bbR/\bbZ} \phi d\mu = \int_{\bbR/\bbZ} ((f^k)'|_{f^{-k}(x)})^{-s} \phi|_{f^{-k}(x)} d\mu
\end{equation}

The main result of the paper is the following:
\begin{theorem}
\label{th-main}
Each multicritical circle map $f$ with irrational rotation number has a unique $s$-measure for any $s<0$. 
\end{theorem}

\begin{remark}
In \S~\ref{sec-atom}, we will show that for multicritical circle maps, $s$-measures are atomless for each $s<0$. This implies that the identity \eqref{eq-def} holds for step functions as well. Standard arguments then imply that it also holds for all measurable functions. 
\end{remark}

We note that the case $s=-1$ is of a particular importance for the study of conjugacy classes due to the following simple observation. Consider a vector field on a circle which satisfies the cohomological equation
$$v(x)=w(f(x))-f'(x)w(x).$$
Equivalently, $v$ corresponds to a conjugacy deformation of $f$:
$$(\id+tw)\circ f\circ (\id+tw)^{-1}=f+tv+o(t).$$
Then 
$$
\int v|_{f^{-1}(x)}d\mu=0
$$
for an automorphic measure $\mu$ with exponent $s=-1$. 
This integral condition  can be used to characterize the tangent bundle to the conjugacy class of a circle map with an irrational rotation number, as seen in \cite[Theorem 2]{DY}. In our paper \cite{GY23} we applied this tool to the Arnold family
$$f_{a,\mu}(x) = x+a+\nu \sin 2\pi x\text{ for }\nu\in \left[0,\frac{1}{2\pi}\right],$$ 
and showed that for each irrational $\alpha$ the $\alpha$-Arnold tongue (or Arnold curve)
$$A_\alpha=\{(a,\nu)\;|\;\rho(f_{a,\nu}(x))=\alpha\}$$
is a $C^1$-smooth curve over the whole closed interval $\nu\in [0,{1}/{2\pi}]$. The same principle was previously used in \cite{Slammert} to show this on the interval without the right endpoint. Furthermore, in \cite{GY23} we showed a similar smoothness result for any irrational Arnold tongue in a typical family which contains both circle diffeomorphisms and maps with a single cubic critical point (critical circle maps).

Our proof of the existence and uniqueness of $(-1)$-measures for critical circle maps in \cite{GY23} essentially relied on the renormalization picture we constructed in \cite{GY}. In this paper we give a much simpler proof for all negative exponents and any number of non-flat  critical points. 
Our proof of uniqueness in Theorem \ref{th-main} follows closely the arguments of \cite{FGN}, but our case turns out to be simpler.  
As an application, in \S~\ref{sec-tongues} we apply our main result to irrational Arnold tongues in the multicritical case, and obtain a corresponding $C^1$-smoothness theorem.

\section{Multicritical circle maps}

As in \cite{FGN}, we define:
\begin{definition}

An orientation-preserving $C^3$-smooth circle homeomorphism $f$ is called \emph{multicritical} if $f'>0$ everywhere except finitely many points, and on a neighborhood of each of these points $c_n$ we have $f(x) = f(c_n) + \psi_n(x) |\psi_n(x)|^{d_n-1}$, where $\psi_n$ is a $C^3$-smooth diffeomorphism.  The number $d_n$ is called the order of the critical point $c_n$. 
\end{definition}

The local expression for $f$ implies the following: 
\begin{lemma}
\label{lem-atcrit}
For every critical point of a multicritical circle map, one can find its neighborhood  $U$  and $c>0$ such that  $f'(x)\le c \frac{|f(I)|}{|I|}$ for any $x\in I\subset U$. 
\end{lemma}

For $\alpha\in(0,1)$ let $\alpha=k_0 + 1/ (k_1 + 1/(k_2 + \dots) \dots )$ be its continued fraction expansion with positive terms; it is unique if and only if $\alpha\notin\bbQ$. We will use the abbreviated notation  $\alpha = [k_0,k_1, k_2, \dots]$. We will let $\ell(\alpha)\in\bbN\cup\{\infty\}$ stand for the minimal length of a continued fraction expansion of $\alpha$. For $n\leq \ell(\alpha)$, we denote $p_n/q_n$ the continued fraction convergents $$p_n/q_n = [k_0, \dots, k_{n-1}].$$
For any point $x\in\bbR/\bbZ$, we let $I_n(x)$ be the arc of the circle with boundary points $x$ and $f^{q_n}(x)$ which does not contain $f(x)$; we denote it $[x,f^{q_n}(x)]$. The iterate $q_n$ is a  closest return time in the sense that $I_n(x)$ does not contain in its interior any iterates $f^j(x)$ with $0<j<q_n$.

We further set $$I_n = [0, f^{q_n}(0)]\text{, and }I_n^k = f^k(I_n).$$
The intervals 
$$\mathcal P_n = \{ I_n^{k}\}_{k=0}^{q_{n+1}-1} \bigcup\{I_{n+1}^{k}\}_{k=0}^{q_n-1} $$
form the {\it $n$-th dynamical partition of the circle}: they have disjoint interiors and cover all of $\bbR/\bbZ$.

 \begin{lemma}[{\bf Real Bounds}]
   \label{lem-realbounds}
   There exists a constant $C_1>0$ such that the following holds. 
For every $C^3$-smooth multicritical circle map $f$  with irrational rotation number, there exists $n_0=n_0(f)\in \bbN$ such that for any $n\geq n_0$, any two adjacent intervals $I,J$ of the partition $\mathcal P_n$ are $C_1$-commensurable:
  $$\frac 1{C_1}<\frac{|I|}{|J|}< C_1.$$
\end{lemma}
For a proof see  \cite{dFbounds}. 

Recall that the \emph{distortion} of a diffeomorphism  $h$ on an interval $I$ is $\sup_{x,y \in I} |h'(x)| / |h'(y)|$.
The following lemma is used to estimate the distortion of high iterates of a circle map. 
An interval $I$ is called $\tau$-\emph{well inside} an interval $J$ if $J = L\cup I \cup R$ where intervals $L, R$ satisfy $|L|/|I|>\tau $ and $|R|/|I|>\tau$.  
\begin{lemma}[{\bf Distortion lemma}]
\label{lem-dis}
For a multicritical circle map $f$ and constants $\tau>0$, $m\in \bbN$, there exists a constant $K=K(f, m, \tau)$ with the following property. 

Let $J\subset I$ be intervals on the circle; let  $f^s(I)= I_{s}$, $f^s(J) = J_s$.  Assume that the intervals $I_s, 0\le s\le n-1$, do not contain critical points of $f$. Also, suppose that the collection of intervals $\{I_s\}_{s=0}^n$ covers each point on the circle at most $m$ times, and suppose  that the interval $J_n$ is $\tau$-well inside  $I_n$.

Then the distortion of $f^n$ on $J$ is bounded by $K$.   
\end{lemma}
The proof can be found in \cite[p.295]{deMeloStrien}; note that for large $n$ the constant can be taken independent of $f$. 

Commensurability of subsequent intervals and distortion estimates, together with Lemma \ref{lem-atcrit}, imply the following.
\begin{lemma}
\label{lem-bounds-deriv}
For any multicritical map $f$  with irrational rotation number, there exists $C_2=C_2(f)$ such that for all $n\in \bbN$, $x\in \bbR/\bbZ$ we have $(f^{q_n}(x))' < C_2$.
\end{lemma}
Again, the constant can be chosen independent of $f$ for $n$ large enough.

\section{Existence of automorphic measures}

In \cite{DY}, the authors proved the following.
\begin{theorem}[R. Douady, J.-C. Yoccoz]
Every $C^2$-smooth circle diffeomorphism with an irrational rotation number has a unique $s$-measure for any real $s$. 
\end{theorem}
This implies the following:
\begin{theorem}
\label{th-exist}
Every multicritical circle map with an irrational rotation number has an $s$-measure for any real $s$.
\end{theorem}
\begin{proof}
Consider a multicritical circle map $f$, and let $f_n\to f$, $\rho(f_n)=\rho(f)$ be a sequence of $C^2$-smooth circle diffeomorphisms that tend to $f$, such that $f_n'\to f' $ uniformly.  

Due to the Douady-Yoccoz theorem, there exists an $s$-measure $\mu_n$ for a circle diffeomorphism  $f_n$. Extracting a weakly converging subsequence $\mu_n\to \mu$, and  passing to the limit in the relation \eqref{eq-def}, we can see that $\mu$ is a $s$-measure for $f$. 
\end{proof}

\section{$s$-measures are atomless}
\label{sec-atom}
\begin {lemma}
\label{lem-atoms}
For a multicritical circle map $f$ with irrational rotation number, for $s<0$, any $s$-measure $\mu_f$ is atomless.

\end{lemma}
\begin{proof}
The proof is similar to  \cite[Lemma 3.8]{GY23} where the statement was proved for $(-1)$-measures.
The definition of $s$-measures implies
\begin{equation}
\label{eq-atoms}
\mu_f(\{p\}) = \mu_f(\{f^k(p)\})\cdot ((f^k)'(p))^{-s}
\end{equation} 
for any $k>0$, $p\in \bbR/\bbZ$. 

If $p$ is not a preimage of a critical point, 
\begin{equation}
\label{eq-meas-series}
1\ge \mu_f ( \{p, f(p), f^2(p), \dots\}) = \mu_f(\{p\}) \cdot \left( 1+\sum_{k=1}^{\infty}((f^{k})'(p))^s\right).
                     \end{equation}
This series diverges due to Lemma \ref{lem-bounds-deriv}, since $s<0$. Thus $\mu_f(\{p\})=0$. 
If $p$ is a preimage of a critical point, \eqref{eq-atoms} immediately implies $\mu_f(\{p\})=0$. 
\end{proof}

\section{Uniqueness of automorphic measures with negative exponents }

We will prove that for $s<0$,  any $s$-measure $\mu$ is ergodic: any invariant measurable set $B$ has measure $0$ or $1$. Since $s$-measures clearly form a convex set, this will imply uniqueness of the $s$-measure. The proof closely follows \cite{FGN}.

Suppose that $B$ is an invariant measurable set and set $$\omega_B(I) = \mu(I\cap B) |I|^{-s}.$$
We will omit the index $B$ until it becomes relevant.

 We will write $ I \succcurlyeq J$ if $\omega(I)>c \omega(J)$ for a constant $c$ that may depend on $f$ and $s$ but not on $B$ or $\mu$. We will write $I\sim J$ if $I \succcurlyeq J$ and $J \succcurlyeq I$. 
 Note that for $I\subset J$, we have $$\omega(I) = \mu(I\cap B) |I|^{-s} \le \mu (J\cap B)|J|^{-s} = \omega(J),$$ thus this function is monotonic; this is a simplification in comparison with the case of  $s$-measures with positive $s$.

For an interval $\Delta\in \mathcal P_n$, let $\Delta^*$ be its union with adjacent intervals of $\mathcal P_{n}$. 
\begin{definition} An integer $0\le k \le q_{n+1}$ is \emph{a critical time} for $\Delta\in \mathcal P_n$ if the interval  $f^k(\Delta^*)$ contains a critical point of $f$.
\end{definition}

Relation \eqref{eq-minus1}, applied to $\phi = \chi_{I\cap B}$, implies that
$$
\mu(f^k(I)\cap B) \max ((f^k)'|_I)^{-s}\ge \mu(I\cap B) \ge  \mu (f^k(I)\cap B) \min ((f^k)'|_I)^{-s}.
$$ 
thus 
\begin{equation}
\label{eq-est}
\omega(f^k(I))\cdot  (\max (f^k)'|_I)^{-s} \cdot \frac{|I |^{-s}}{|f^k(I)|^{-s}}\ge \omega(I) \ge  \omega (f^k(I))\cdot  \min ((f^k)'|_I)^{-s} \cdot \frac{|I |^{-s}}{|f^k(I)|^{-s}}.
\end{equation}

\begin{lemma}
\label{lem-est}
For any multicritical circle map $f$ with irrational rotation number, for sufficiently large $n>n_0(f)$, for any $\Delta\in \mathcal P_n$ and its images $f^k(\Delta)$, $f^l(\Delta)$ with $0\le k,l\le q_{n+1}$, if there is no critical time for $\Delta$ between $k$ and $l$, then $f^k(\Delta)\sim f^l(\Delta)$. 

If $0\le k \le q_{n+1}$ is a critical time, then $f^{k+1}(\Delta) \succcurlyeq  f^k(\Delta)$.
\end{lemma}
\begin{proof}
The first statement follows from the Distortion lemma (Lemma \ref{lem-dis}) and \eqref{eq-est}. 

For the second statement, let $I=f^k(\Delta)$. On a small neighborhood $U$ of each critical point, we have an upper estimate on the derivative $f'(x)\le c \frac{|f(I)|}{|I|}$ for any $x\in I\subset U$ due to Lemma \ref{lem-atcrit}. If $n>n_0(f)$ and $k$ is a critical time, then the segment  $f^k(\Delta^*)$ is small enough to belong to this neighborhood. 
Thus \eqref{eq-est} implies for $I= f^k(\Delta)$: 
$$
c^{-s} \omega(f(I))\ge \omega(f(I)) \cdot  \max (f'|_I)^{-s} \cdot \frac{|I |^{-s}}{|f(I)|^{-s}} \ge \omega(I)  \qquad \text{ i.e. } \qquad f(I) \succcurlyeq  I
$$
\end{proof}

Note that the inequality $ \succcurlyeq$ in the last formula points in the other direction than for $s$-measures with positive $s$ in \cite{FGN}.  
However, an analogue of Theorem 5.12 of \cite{FGN} holds true:
\begin{theorem}
\label{th-omega}
For any multicritical circle map $f$ with irrational rotation number, for sufficiently large $n>n_0(f)$, for any  $\Delta_1, \Delta_2 \in \mathcal P_n$, we have:
\begin{itemize}
 \item  If $\Delta_1$ is a long interval and $\Delta_2$ is a short interval of $\mathcal P_n$, then
 $\Delta_1 \succcurlyeq\Delta_2;$
 \item If $\Delta_1, \Delta_2 $ are both long intervals or both short intervals of $\mathcal P_n$, then 
$\Delta_1\sim \Delta_2$.

 \end{itemize}
 
\end{theorem}
\begin{proof}
Let $\Delta_1$ be long and $\Delta_2$ be short. Use Lemma \ref{lem-est} along a sequence of iterates of the interval $I_n=[0, f^{q_n}(0)]$ that contains $\Delta_1$. Note that if we have $d$ critical points, then we have at most $3d$ critical times in this chain, since $\bigcup_{\Delta \in \mathcal P_n} \Delta^*$ covers the circle at most three times. Thus $$\Delta_1 \succcurlyeq  I_n.$$ Similarly, $$I^{q_{n}}_{n+1}  \succcurlyeq \Delta_2.$$ Since $I_n\supset I^{q_n}_{n+1}$, we have $$\omega(I_n)\ge \omega(I^{q_n}_{n+1})\text{ and thus  }\Delta_1  \succcurlyeq \Delta_2.$$   

Let $\Delta_1, \Delta_2$ be long intervals. Similarly, we have  $$\Delta_1 \succcurlyeq  I_n\text{ and }I^{q_{n+1}}_{n}  \succcurlyeq \Delta_2 .$$ We have $I_n \succcurlyeq I_{n+1}$ due to the first statement of the theorem, the intervals $I_n$ and $I_{n+1}$ have comparable length, and $I_n\cup I_{n+1}\supset I^{q_{n+1}}_{n}$, which implies $$I_n \succcurlyeq I^{q_{n+1}}_{n}\text{ and thus }\Delta_1 \succcurlyeq \Delta_2.$$

Since long intervals of $\mathcal P_n$ are short intervals of $\mathcal P_{n+1}$, the statement for short intervals follows.  

\end{proof}

\begin{theorem}
For any multicritical circle map $f$ with irrational rotation number and $s<0$, any $s$-measure $\mu$ is ergodic.
\end{theorem}
\begin{proof}
Assume the contrary, and let $B$ be an invariant set of an intermediate measure.  
Due to the Lebesgue density point theorem,  for any $\eps>0$, we can choose $x$ such that 
$$
\frac{\mu(I(x)\cap B)}{\mu(I(x))}<\eps 
$$
for any sufficiently short interval $x\in I(x)$ such that it has bounded eccentricity. Let $\Delta_n\in \mathcal P_n$ be such that $x\in \Delta_n$; due to commensurability of neighboring intervals of $\mathcal P_n$, we can take $\Delta^*_n$ to be $I(x)$, and thus, for sufficiently large $n$,
$$
\mu(\Delta^*_n\cap B) |\Delta^*_n|^{-s}< \eps \mu (\Delta^*_n) |\Delta^*_n|^{-s},
$$
i.e. $$\omega_B(\Delta^*_n) < \eps \, \omega_{S^1}(\Delta^*_n).$$

Let $\Delta^*_n = J_1 \cup J_2 \cup J_3$ where $J_1$ is a long interval of $\mathcal P_n$. Theorem \ref{th-omega} implies $$J_1 \succcurlyeq J_2, \;J_1 \succcurlyeq J_3$$ with respect to $\omega_{S^1}$; note also that these intervals have commensurable lengths since they are neighbors in the partition  $\mathcal P_n$. Together with monotonicity of $\omega_B$, this implies $$\omega_B(J_1) < C\, \eps \, \omega_{S^1}(J_1)$$
for a uniform $C=C(f, s)$ and a long interval $J_1\in \mathcal P_n$.  
Theorem \ref{th-omega}, applied to $\omega_B$ and $\omega_{S^1}$,  gives
$$\omega_B(\Delta) < \tilde C\,\eps \,\omega_{S^1}(\Delta)$$
for a uniform $\tilde C=\tilde C(f, s)$ and for any long interval $\Delta\in \mathcal P_n$,  thus 
$$
\mu(\Delta\cap B) < \tilde C\,\eps\, \mu(\Delta).$$ Since the same arguments apply to $\mathcal P_{n+1}$, the inequality holds for any short interval   $\Delta\in \mathcal P_n$. 
Since these intervals cover the circle, adding these inequalities, we get 
$$\mu(B) < \tilde C\, \eps \, \mu(S^1)=\tilde C\, \eps,$$
and since this holds for any $\eps$ with uniform $\tilde C$, we have $\mu(B)=0$.  

The contradiction shows that $\mu$ is ergodic.

\end{proof}

As mentioned above, since $s$-measures form a convex set, ergodicity of every $s$-measure implies uniqueness of $s$-measures for multicritical circle maps. Together with existence of $s$-measures (Theorem \ref{th-exist}), this completes the proof of Theorem \ref{th-main}.

\section{Differentiability  of limits in Banach spaces}

The next theorem is a suitable version of a standard result: uniform convergence of derivatives, with pointwise convergence of functions, implies differentiability of the limit.  Recall that a star domain is a set that contains zero and includes an interval $[0, w]$ together with each its point $w$.
\begin{theorem}
\label{th-banach}
Let $V$ be a Banach space, let $S\subset V$ be an open subset such that $\overline S$ is a star domain. 
Suppose that  continuous functions $\tau_n\colon \overline S \to \bbR$ converge pointwise on $\overline S$  to a continuous function $\tau\colon \overline S\to \bbR$. Assume that each function $\tau_n$ is Gateaux differentiable on a neighborhood of $\overline S$,  with differential  at any point $v$ being a bounded linear functional $d|_v \tau_n$. Suppose that there exists a bounded linear functional $L$ such that for any sequence $v_n\to 0$, $v_n\in \overline S$, the linear functionals $d|_{v_n}\tau_n$ converge in the operator norm to $L$. 

Then $\tau$ is Frechet differentiable at $0$, with derivative equal to $L$.  
\end{theorem}
\begin{proof}
For any $\eps>0$, find a ball $B_\eps$ centered at zero and $N_\eps\in \bbN$ such that  
\begin{equation*}
\| d|_{v} \tau_n -  L\|<\eps \qquad \text{ for any }v\in B_\eps\cap \overline S, \, n>N_\eps.
\end{equation*}
 Then we have  due to Gateaux differentiability of $\tau_n$ on $[0,v]\subset \overline S$:
 $$\tau_n (v)= \tau_n(0) +  \int_{0}^1 d|_{tv} \tau_n (v) dt,$$
 thus
 \begin{equation*}
 \left|\tau_n(v) - \tau_n(0) -  L v \right| \le  \eps \|v\|  \qquad \text{ for any }v\in B_\eps\cap \overline S, \, n>N_\eps.
 \end{equation*}
Passing to the limit as $n\to \infty$, we get 
 \begin{equation*}
 \left|\tau(v) - \tau(0) -  L v \right| \le  \eps \|v\|  \qquad \text{ for any }v\in B_\eps\cap \overline S.
 \end{equation*}
Thus $\tau$ is Frechet differentiable at $0$, with derivative given by $L$.
\end{proof}

The following theorem shows that weak convergence of measures implies convergence of the corresponding linear functionals in the Banach space of $C^1$-smooth functions (in the operator norm). It will be used to establish convergence of linear functionals, as needed for Theorem \ref{th-banach}. 
\begin{definition}
The Kantorovich-Rubinstein distance (a.k.a. Wasserstein distance $W_1$) between probability measures on $\bbR/\bbZ$ is given by $$\distkr (\mu, \nu) = \sup_{v\in 1\text{-\rm Lip}} \int v (d\mu-d\nu)$$
where $1\text{-\rm Lip}$ is the set of Lipschitz functions with Lipschitz constant $1$. 
\end{definition}
\begin{theorem}
\label{th-KR}
\begin{itemize}
\item The Kantorovich-Rubinstein distance metrizes weak convergence on $\bbR/\bbZ$:  convergence of probability measures in this metric is equivalent to weak convergence. 
\item If $\distkr(\mu_n, \mu)\to 0$ then in the Banach space of $C^1$-smooth vector fields on $\bbR/\bbZ$,  linear functionals defined by $\mu_n  $ converge in the operator norm to the  linear functional defined by $\mu$.
\end{itemize}
\end{theorem}
See \cite[Sec. 6]{Villani} for the proof of the first statement. The second statement follows from the definition.

\section{(-1)-measures and smoothness of irrational Arnold tongues}
\label{sec-tongues}

Let $\mu_f$ be the (-1)-measure of a circle homeomorphism $f$; if $f$ is a $C^2$-smooth diffeomorphism or a $C^3$-smooth multicritical map with irrational rotation number, this measure exists and is unique due to the above. The following result is \cite[Theorem 2]{DY}. 
\begin{theorem}
\label{th-tangent}
For any irrational $\alpha$, in the Banach manifold of $C^2$-smooth orientation-preserving circle homeomorphisms of the circle with the norm
$$\|f\|_{C^2}=\sum_{i=0,1,2}\|f^{(i)}\|_\infty,$$
the set $\mathcal M_\alpha =\{f\in \mathcal D \mid \rho (f) = \alpha\}$ at each its point $f_0$ has a tangent subspace in the sense of Gateaux, given by 
\begin{equation}
\label{eq-tangent}
v\in T_{f_0}\mathcal M_{\alpha}\Leftrightarrow \int_{\bbR/\bbZ} v|_{f_0^{-1}(x)} d\mu_{f_0}=0,
\end{equation} 
and this tangent space varies continuously with $f_0$ in the weak topology.
\end{theorem}
We will prove the analogue of this result for multicritical circle maps, for Frechet derivatives instead of Gateaux derivatives.

Let $\mathcal D$ be a Banach manifold of $\bbR/\bbZ$-preserving orientation-preserving $C^3$-smooth maps with the norm
$$\|f\|_{C^3}=\sum_{i=0,1,2,3}\|f^{(i)}\|_\infty.$$ Circle diffeomorphisms form an open subset $\mathcal B$ in $\mathcal D$; multicritical circle maps form a dense subset in the boundary of this open subset. Rotation number is defined on $\overline {\mathcal B}$, and any diffeomorphism or a  multicritical circle map from $ \overline {\mathcal B}$ has a unique (-1)-measure.
\begin{theorem}
\label{th-tangent-cr}
For any irrational $\alpha$, the set $\mathcal M_\alpha =  \{f\in \mathcal D \mid \rho (f) = \alpha\}$ at each its point $f_0\in \overline {\mathcal B}$ which is either a diffeomorphism or a multicritical circle map has a tangent subspace in $\overline{\mathcal B}$ in the sense of Frechet, given by \eqref{eq-tangent}.

Moreover, this tangent space varies continuously with $f_0$ in the weak topology. \end{theorem}
\begin{proof}

Let $f_0\in \overline {\mathcal B}$, and assume that $f_0$ is a multicritical circle map with $\rho(f)=\alpha\in \bbR\setminus \bbQ$. 
Fix small $c>0$ so that we can find circle diffeomorphisms $f_n$ such that $f_n\to f_0$ in $\mathcal D$ and 
 \begin{equation}
\label{eq-c}
f_n'>f_0'\text{ whenever }f_0'<c. 
\end{equation}
Let $V_0 = \{v(0)=0\}$ be the subspace of $T\mathcal D$. Represent $\mathcal M_\alpha$ as a graph of a continuous function $\tau(v)\colon V_0\to \bbR$: let $\mathcal M_\alpha$ be given by $\{f_0+ v+\tau(v), v\in V_0\}$. Then $\tau(v)$ is well-defined and continuous whenever $f_0+v$ is a circle homeomorphism, since the rotation number in the family $t\to f_0+v+t$ is monotonic and strictly monotonic at the points with irrational rotation numbers. Also, we have  $\tau(0)=0$. Let $$S = \{v \in V_0   \mid |v'|<c/2 \text{ and } f_0+v \text{ is a circle diffeomorphism}\},$$  $$\overline S =\{v \in V_0  \mid |v'|\le c/2 \text{ and } f_0+v \text{ is a circle homeomorphism}\}.$$ Due to the above, $\tau$ is continuous on $\overline S$; it is easy to see that $\overline S$ is a  star-shaped region. 

Let $f_n$ be as above. Let 
\begin{equation}
\label{eq-tau}\tau_n(v) =  f_0(0)-f_n(0) + \tau(v +f_n-f_0 - f_n(0)+f_0(0)), 
\end{equation}
so that $\mathcal M_\alpha$ is given by $\{f_n+ v+\tau_n(v), v\in V_0\}$. 

Due to Theorem \ref{th-tangent}, $\tau_n$ is Gateaux differentiable at any point $v$ such that $f_n + v$ is a circle diffeomorphism. 
This open set  contains $\overline S$ for large $n$: indeed,  if $f_0+v$ is a circle homeomorphism, then  $f_0'+v'\ge 0$. Thus at the points with $f_0'<c$, we have $f_n'+v' > f_0'+v'\ge 0 $ due to \eqref{eq-c}. At the points with  $f_0'\ge c$, we have $f_n'+v'>0 $ for large $n$ due to the inequality $|v'|\le c/2$. We conclude that $\tau_n$ is Gateaux differentiable on a neighborhood of $\overline S$. Also, $\tau_n\to \tau$ on $\overline S$ due to \eqref{eq-tau}.

The derivative of $\tau_n$ at $v$ is given by \eqref{eq-tangent} for the circle map $f_n+v+\tau_n(v)$. We will be able to use Theorem \ref{th-banach} as soon as we prove that for any $v_n\to 0$, linear functionals \eqref{eq-tangent} for $g_n = f_n+v_n+\tau_n(v_n)$ converge  to the linear functional \eqref{eq-tangent} for $f_0$ in the operator norm.  Note that $g_n\to f_0$, since $f_n\to f_0, v_n\to 0$, and the rotation number is continuous. 

 Let $\tilde \mu_f$ be given by $\tilde \mu_f (I ) = \mu_f(f(I))$. Linear functionals in \eqref{eq-tangent} are given by $v\mapsto \int_{\bbR/\bbZ} v d\tilde \mu_f$. For any sequence $g_n\to f_0$ in $\mathcal M_\alpha$, any weakly convergent subsequence of a sequence of measures $\mu_{g_n}$ converges to a (-1)-measure for $f_0$; since this measure is unique, we have $\mu_{g_n}\to \mu_{f_0}$ weakly. Thus $\tilde \mu_{g_n}\to \tilde \mu_{f_0}$ weakly. This  implies convergence with respect to the Kantorovich-Rubinstein metric, thus (due to Theorem \ref{th-KR})  convergence of the corresponding linear functionals in the operator norm.
 
Theorem \ref{th-banach} now implies the first statement of the theorem for multicritical $f\in \mathcal {\overline B}$.

In the case $f\in \mathcal B$ (that is, when $f$ is a circle diffeomorphism), the proof is analogous but simpler: we can take $S$ to be a neighborhood of zero and $\tau_n=\tau$.  

Due to the same argument as above, for any sequence $g_n\to g$ where $g_n,g \in \mathcal M_\alpha$ are either diffeomorphisms or multicritical circle maps, we have $\tilde \mu_{g_n} \to \tilde \mu_{g}$ weakly. This  implies the second statement of the theorem.
\end{proof}

Let us note that we can consider the same construction in the Banach manifold ${\mathcal W}_U$ of bounded real-symmetric analytic maps in an 
annulus $U\subset \bbC/\bbZ$ around the circle $\bbR/\bbZ$. Denoting by $\mathcal A\subset {\mathcal W}_U$ the open set of diffeomorphisms of the circle, we see that every map in the boundary of $\mathcal A$ is multicritical. The same statement as above holds in this case for $$\mathcal M^A_\alpha=\{f\in \overline{\mathcal A},\;\rho(f)=\alpha\notin\bbQ\}.$$ We note that in the analytic context it can be shown that $\mathcal M^A_\alpha$ has a smooth Banach submanifold structure (compare with \cite{EsYam}).

 Finally,  we formulate a corollary of Theorem \ref{th-tangent-cr} for finite-parameter families of circle maps. 
Let $\lambda = (\lambda_1,\dots, \lambda_r)$ be a multidimensional parameter. 
For a family $f_\lambda$ of circle homeomorphisms, \emph{ Arnold $\alpha$-tongue} is the set 
$$
\{\lambda \mid \rho(f_{\lambda})=\alpha\}
$$
For rational $\alpha$, Arnold $\alpha$-tongues may have non-empty interiors. Assume that the family $f_\lambda$ is smooth and monotonic (i.e. $\frac {df}{d\lambda_1}>0$). Since in this case, the rotation number is strictly monotonic at irrational values, irrational tongues in such a family are graphs of continuous functions on $(\lambda_2,\dots, \lambda_n)$.

Consider a $C^1$-smooth family of $C^3$-smooth  circle maps $\lambda\to f_\lambda$. Assume that the family is monotonic with respect to $\lambda_1$:  $$\frac {df}{d\lambda_1}>0,$$ and assume that $f_\lambda$ is a  multicritical circle map for $\lambda_2=0$ and a circle diffeomorphism for $\lambda_2<0$. An example is a standard Arnold's family with parametrization $f_{\lambda}(x) = x+\lambda_1 + \left(\frac 1{2\pi}+\lambda_2\right)\sin 2\pi x$. 
\begin{corollary}
For a family $f_\lambda$ as above, for any irrational $\alpha$, the Arnold $\alpha$-tongue is given by a  function $\lambda_1 = a(\lambda_2,\dots, \lambda_n)$ that is continuously differentiable at $0$. 
Its derivative in the direction of $w\in \bbR^{n-1}$ satisfies
\begin{equation}
\label{eq-deriv}
L_w a \cdot \int_{\bbR/\bbZ} \frac{df_\lambda}{d\lambda_1} (f^{-1}(x)) d\mu_{f} = -\int_{\bbR/\bbZ} L_w f_{\lambda}|_{(f^{-1}(x))} d\mu_f.
\end{equation}

\end{corollary}
\begin{proof}
The Arnold $\alpha$-tongue is the intersection of the set $\{f \in \mathcal D \mid \rho( f) =\alpha\}$ with the smooth family $f_\lambda$. The family is transversal to this tangent space due to monotonicity: $$\int_{\bbR/\bbZ} \frac{df}{d\lambda_1}  (f^{-1}(x)) d\mu_f >0.$$ Thus the tangent space to the $\alpha$-tongue is the intersection of the tangent space \eqref{eq-tangent} with the tangent space to our family. This gives \eqref{eq-deriv}. 
\end{proof}

\bibliographystyle{amsalpha}
\bibliography{biblio}
\end{document}